\documentclass[12pt]{amsart}
\usepackage[arrow,dvips,matrix,arrow,ps,color,line,curve,frame]{xy}
\usepackage{amscd,amssymb}
\usepackage{graphicx}

\let\oldlabel=\label
\def\prellabel{\marginparsep=1em
    \def\label##1{\oldlabel{##1}\ifmmode\else\ifinner\else
         \marginpar{{\footnotesize\ \\ \tt
                    ##1}}\fi\fi}}

\let\Bbb=\mathbb

\def\width{\operatorname{width}}
\def\vertex{\operatorname{vert}}

\def\CN{\operatorname{CN}}
\def\BCN{\operatorname{BCN}}

\def\cn{\operatorname{\mathfrak{cn}}}
\def\bcn{\operatorname{\mathfrak{bcn}}}

\def\E{\operatorname{\textsf{E}}}

\def\aff{\operatorname{aff}}

\def\inte{\operatorname{int}}

\def\Q{{\Box\kern1pt}}
\def\rank{\operatorname{rank}}
\def\inte{\operatorname{int}}

\def\conv{\operatorname{conv}}%

\def\RR{{\Bbb R}}
\def\FF{{\Bbb F}}
\def\QQ{{\Bbb Q}}
\def\ZZ{{\Bbb Z}}
\def\NN{{\Bbb N}}

\def\UU{\Bbb U_{\vertex}(P,c)}

\newtheorem{lemma}{Lemma}[section]
\newtheorem{corollary}[lemma]{Corollary}
\newtheorem{theorem}[lemma]{Theorem}

\theoremstyle{definition}
\newtheorem{definition}[lemma]{Definition}
\newtheorem{remark}[lemma]{Remark}

\begin{document}

\title[Convex normality]{Convex normality of rational polytopes\\ with long edges}

\author[Joseph Gubeladze]
{Joseph Gubeladze}

\thanks{Partially supported by NSF grant DMS-1000641}

\address{Department of Mathematics, San Francisco
State University, San Francisco, CA 94132, USA}


\begin{abstract}
We introduce the property of \emph{convex normality} of rational
polytopes and give a dimensionally uniform lower bound for the
edge lattice lengths, guaranteeing the property. As an
application, we show that if every edge of a lattice $d$-polytope
$P$ has lattice length $\ge4d(d+1)$ then $P$ is normal. This
answers in the positive a question raised in 2007. If $P$ is a
lattice simplex whose edges have lattice lengths $\ge d(d+1)$ then
$P$ is even covered by lattice parallelepipeds. For the approach
developed here, it is necessary to involve rational polytopes even
for the application to lattice polytopes.
\end{abstract}

\maketitle

\section{Integrally closed polytopes}\label{intro}
All our polytopes are assumed to be convex. For a polytope $P$ the
set of its vertices will be denoted by $\vertex(P)$.

A polytope $P\subset\RR^d$ is \emph{lattice} if
$\vertex(P)\subset\ZZ^d$, and $P$ is \emph{rational} if
$\vertex(P)\subset\QQ^d$.

Let $P\subset\RR^d$ be a lattice polytope and denote by $L$ the
subgroup of $\ZZ^d$, affinely generated by the lattice points in
$P$; i.~e.,
$$
L=\sum_{x,y\in P\cap\ZZ^d}\ZZ(x-y)\subset\ZZ^d
$$

\begin{definition}\label{definition}(\cite[Def. 2.59]{kripo})
Let $P\subset\RR^d$ be a lattice polytope.

(a) $P$  is \emph{integrally closed} if the following condition is
satisfied:
\begin{align*}
c\in\NN,\ z\in cP\cap\ZZ^d\quad\Longrightarrow\quad\exists
x_1,\ldots,x_c\in P\cap\ZZ^d\quad x_1+\cdots+x_c=z.
\end{align*}

(b) $P$ is \emph{normal} if  for some (equivalently, every) point $t\in P\cap\ZZ^d$
the following condition is satisfied:
\begin{align*}
c\in\NN,\ z\in cP\cap(ct+L)\quad \Longrightarrow\quad \exists
x_1,\ldots,x_c\in P\cap\ZZ^d,\quad x_1+\cdots+x_c=z.
\end{align*}
\end{definition}
\noindent(Observe, $P\cap(t+L)=P\cap\ZZ^d$)

The normality property is invariant under \emph{affine
isomorphisms of lattice polytopes,} and the property of being
integrally closed is invariant under affine changes of
coordinates, leaving the lattice structure $\ZZ^d\subset\RR^d$
invariant.

A lattice polytope  $P\subset\RR^d$ is integrally closed if and
only if it is normal and $L$ is a direct summand of $\ZZ^d$.
Obvious examples of normal but not integrally closed polytopes are
the s.~c. \emph{empty} lattice simplices of large volume. No
classification of such simplices is known in dimensions $\ge4$,
the main difficulty being the lack of satisfactory
a characterization of their lattice widths; see
\cite{empty2,empty1}. For recent advances in the field see
\cite{kantor,batyrev}.

A normal polytope $P\subset\RR^d$ can be made into a
full-dimensional integrally closed polytope by changing the
lattice of reference $\ZZ^d$ to $L$, the ambient Euclidean
space $\RR^d$ to the subspace $\RR L$, and shifting $P$ so that $0\in P$. In particular, normal and
integrally closed polytopes refer to same isomorphism classes of
lattice polytopes. In the literature, however, the difference
between `normal' and `integrally closed' is sometimes blurred.

Normal/integrally closed polytopes enjoy popularity in algebraic
combinatorics and they have been showcased on recent workshops
(\cite{aim,mfo}). These polytopes represent the homogeneous case
of the Hilbert bases of finite positive rational cones and the
connection to algebraic geometry is that they define projectively
normal embeddings of toric varieties. There are many challenges of
number theoretic, ring theoretic, homological, and $K$-theoretic
nature, concerning the associated objects: Ehrhart series',
rational cones, toric rings, and toric varieties; see
\cite{kripo}.

If a lattice polytope is covered by (in particular, subdivided
into) integrally closed polytopes, then it is integrally closed as
well. The simplest integrally closed polytopes one can think of
are \emph{unimodular simplices}, i.~e., the lattice simplices $
\Delta=\conv(x_1,\ldots,x_k)\subset\RR^d$, $\dim\Delta=k-1$, with
$x_1-x_j,\ldots,x_{j-1}-x_j,x_{j+1}-x_j,\ldots,x_k-x_j$ a part of
a basis of $\ZZ^d$ for some (equivalently, every) $j$.

Unimodular simplices are the smallest `atoms' in the world of
normal polytopes. But not all $3$-dimensional integrally closed polytopes
are triangulated into unimodular simplices \cite{kantor-sarkaria}. (The first such example  in dimension 4 was given in \cite[Prop.
1.2.4]{brgutr}.) Moreover, not all $5$-dimensional integrally closed
polytopes are covered by unimodular simplices \cite{unico} --
contrary to what had been conjectured before \cite{sebo}. Further
`negative' results, such as \cite{icp} and \cite{cara}, the latter disproving
an additive version of the unimodular cover property that was conjectured in \cite{cfs}, contributed to the current thinking
in the area that there is no succinct geometric characterization
of the normality property. One could even conjecture that in
higher dimensions the situation gets as bad as it can; see the
discussion at the end of \cite[p. 2313]{mfo}.

`Positive' results in the field mostly concern special classes of
lattice polytopes that are normal, or have unimodular
triangulations or unimodular covers. Knudsen-Mumford's classical
theorem (\cite[Sect. 3B]{kripo}, \cite[Chap. III]{kkms}) says
that every lattice polytope $P$ has a multiple $cP$ for some
$c\in\NN$ that is triangulated into unimodular simplices. Whether the factor $c$ can be chosen uniformly w.r.t. dimension
seems to be a very hard problem. More recently, it was shown in
\cite{multiples} that there exists a dimensionally uniform
exponential lower bound for unimodularly covered dilated
polytopes. By improving one crucial step in \cite{multiples}, von
Thaden was able to cut down the bound to a degree 6 polynomial
function in the dimension \cite[Sect. 3C]{kripo},
\cite{vonthaden}.

For polytopes,  arising in a different context and admitting
unimodular triangulations as certificate of normality, see
\cite{beck-hosten,kitamura,ohsugi-hibi,reiner-welker}; for other techniques
for establishing normality, along with its higher homological analogues,
see \cite{payne}.

The results above on dilated polytopes yield no new examples of
normal polytopes, though. In fact, an easy argument ensures that
for any lattice $d$-polytope $P$ all multiples $cP$,
$c\ge d-1$, are integrally closed \cite[Prop.
1.3.3]{brgutr}, \cite{ewald}. However, that argument does not
allow a modification that would apply to lattice polytopes with
long edges of independent lengths.

The following conjecture was proposed in \cite[p. 2310]{mfo}:

\medskip\noindent{\bf Conjecture.} \emph{Simple lattice polytopes with
`long' edges are normal, where `long' means some invariant,
uniform in the dimension.}

\emph{More precisely, let $P$ be a simple lattice polytope. Let
$k$ be the maximum over the heights of Hilbert basis elements of
tangent cones to vertices of $P$. Then, if any edge of $P$ has
length $\ge k$, the polytope $P$ should be normal.}

\medskip Here: (i) the length is measured in the lattice sense, (ii)
`tangent cones' is the same as corner cones, and (iii) the heights
of Hilbert basis elements of corner cones are normalized w.r.t.
the extremal generators of the cones (leading, in particular, to
non-integral rational heights).

The second part of the conjecture is a far reaching extension of
the following well known problem, a.k.a. \emph{Oda's question},
that has attracted much interest recently: are all \emph{smooth
polytopes} normal? A lattice polytope $P\subset\RR^d$ is called
smooth if the primitive (i.~e., with coprime components) edge
vectors at every vertex of $P$ define a part of a basis of
$\ZZ^d$. Smooth polytopes correspond to the projective
embeddings of smooth projective toric varieties and they are simple polytopes
with $k=1$. Oda's question still remains wide open. The
fact that so far no smooth polytope just without a unimodular
triangulation has been found illustrates how limited our
understanding in the area is. The second part of the conjecture
yields also a dimensionally uniform bound, mentioned in the
first part. In fact, it is known that, for every $d\ge2$, the normalized heights of
Hilbert basis elements of a simplicial rational $d$-cone
are at most $d-1$; see, for instance, \cite[Prop.
2.43(d)]{kripo}.

Another motivation for the conjecture above is the following question in toric geometry, discussed
in \cite[p. 2310]{mfo}: are all line bundles over a projective toric variety,
deep enough inside the \emph{nef cone}, projectively normal? If so, how deep is `deep enough'? See \cite[Chapter 6]{toricbook} for generalities on the nef cones of toric varieties. It is not difficult to show that if an ample line bundle $\mathcal L$ over a projective toric variety $X$ is on lattice depth $l$ inside the nef cone $\text{Nef}(X)$ w.r.t. every facet of the cone, then the edges of the lattice polytope $P$ of $\mathcal L$ are all of lattice lengths $\ge l$.

In this paper we introduce the notion of
\emph{$k$-convex-normality}, $k\in\QQ_{\ge2}$, which is a
`convex-rational' version of Definition \ref{definition}. Next is
the main result of the paper:

\begin{theorem}\label{conjecture(a)}
Let $P$ be a rational (not necessarily simple) polytope of dimension $d$ whose every edge has lattice length
$\ge d(d+1)k$. Then $P$ is $k$-convex-normal.
\end{theorem}

Although $k$-convex-normality concerns the dilated polytopes $cP$
with $c\in[2,k]_\QQ$, when applied to lattice polytopes this is
enough to cover the factors $c\in\NN$ in Definition
\ref{definition}, even with $k=4$. As an application to lattice
polytopes, we prove the first part of the conjecture above in the
following strong form:

\begin{theorem}\label{latticepolytopes}
Let $P$ be a (not necessarily simple) lattice polytope of
dimension $d$.

(a) If every edge of $P$ has lattice length $\ge4d(d+1)$ then
$P$ is integrally closed.

(b) If $P$ is a simplex and every edge of $P$ has lattice length
$\ge d(d+1)$ then $P$ is covered by lattice parallelepipeds. In
particular, $P$ is integrally closed.
\end{theorem}

In particular, if a line bundle $\mathcal L$ over a projective toric variety is on lattice depth $\ge4d(d+1)$ w.r.t. every facet of the nef cone, then $\mathcal L$ is projectively normal.

\medskip For the reader's convenience we now give a brief outline
of the proof of Theorem \ref{conjecture(a)}. Let $P$ be a rational polytope with long
edges. Assuming Theorem \ref{conjecture(a)} is true in dimension
$d-1$, we first show that the neighborhood of a certain width of
the boundary surface of any multiple $cP$ with $c\in[2,k]_\QQ$
behaves as if $P$ were convex-normal. Then it is shown that the
complement of this neighborhood is covered by certain parallel
translates of lattice parallelepipeds inside $cP$. This does not
require the inductive assumption and is achieved by propagating
`corner parallelepipedal covers' deep inside $cP$. Actually, the
situation is more subtle, the reason being that the width of the mentioned boundary of $cP$ depends on $P$ and does not grow along with $c$. As a result, one needs that the inductively covered boundary neighborhood and
the region, covered by the parallelepipeds, overlap in certain
nontrivial way.

\medskip\noindent\emph{Acknowledgement.} We thank (i) the referee, whose thorough study of the paper resulted in a number of substantial expositional improvements and inclusion of pictures, and  (ii) Diane Fenster, who created the pictures used in this paper.

\subsection{Notation and terminology}\label{preliminaries}
The affine and convex hulls of a subset $X\subset\RR^d$ will be
denoted, respectively, by $\aff(X)$ and $\conv(X)$.

The relative interior $\inte(P)$ of a polytope $P\subset\RR^d$ is
by definition the absolute interior of $P$ in $\aff(P)$.

Let $H\subset\RR^d$ be an affine hyperplane. The one of the two half-spaces, bounded by $H$ and clear from the context, will denoted by $H^+$.

For a polytope $P\subset\RR^d$ the set of its facets will be denoted by $\FF(P)$. (Recall, $\vertex(P)$ is the set of vertices of $P$.) If $\dim P=d$ and $F\in\FF(P)$, the half-space
$H_F^+$ and hyperplane $H_F$ are defined from the unique
irredundant representation of the form (\cite[Thm. 1.6]{kripo},
\cite[Thm. 2.15(7)]{ziegler})

$$
P=\bigcap_{\FF(P)} H_F^+,\quad H_F=\aff(F).
$$

A polytope is \emph{simple} if its edge directions at every vertex
are linearly independent.

A \emph{parallelepiped} is by definition the Minkowski sum of
segments of linearly independent directions.

Cones $C\subset\RR^d$ are always assumed to be finite and
positive, i.~e., they are intersections of finitely many
homogeneous half-spaces and contain no nontrivial subspaces. A
cone is \emph{simplicial} if its edge directions are
linearly independent.

Let $C\subset\RR^d$ be a rational cone, i.~e., $C$ is the
intersection of half-spaces with rational boundary hyperplanes.
Then the primitive lattice points on the one-dimensional faces (\emph{rays}) of $C$ are
called the \emph{extremal generators} of $C$.

A \emph{$d$-polytope} or \emph{$d$-cone} is the same as a
$d$-dimensional polytope or, respectively, $d$-dimensional cone.

$\RR_+$, $\QQ_+$, and
$\ZZ_+$ refer to the corresponding sets of nonnegative numbers.

For an interval $I\subset\RR$ and number $\lambda\in\RR$  we let
\begin{align*}
I_\QQ=I\cap\QQ,\quad
I_\NN=I\cap\NN,&\quad\QQ_{\ge\lambda}=[\lambda,\infty)_\QQ,\quad\QQ_{>\lambda}=(\lambda,\infty)_\QQ,\\
&\quad\NN_{\ge\lambda}=[\lambda,\infty)_\NN,\quad\NN_{>\lambda}=(\lambda,\infty)_\NN.
\end{align*}

For a subset $X\subset\RR^d$ we put $\RR_+X=\{\lambda x\ |\
\lambda\in\RR_+,\ x\in X\}$.

The \emph{lattice length} of a rational segment
$[x,y]\subset\RR^d$, $x,y\in\QQ^d$, is the ratio of its Euclidean length and that
of the primitive integer vector in the direction of $y-x$.

For a rational polytope $P$, by $\E(P)$ we denote the minimum of
the lattice lengths of the edges of $P$.

The Euclidean distance between a point $x\in\RR^d$ and an affine hyperplane
$H\subset\RR^d$ is denoted by $\|x,H\|$.

\section{Convex normality}\label{result}

For a polytope $P\subset\RR^d$ and a rational number $c\ge1$ denote
$$
\UU=\bigcup_{ {\tiny\begin{matrix}
&v\in\vertex(P)\\
&x\in(c-1)P\bigcap\left((c-1)v+\ZZ^d\right)
\end{matrix}
}} x+P.
$$

\noindent Obviously, $\UU\subset cP$.

Crucial in our approach to the normality property is the following notion that mixes just the optimal amounts of discreteness and continuity:

\begin{definition}\label{convexnormal}
Assume $d\in\NN$, $k\in\QQ_{\ge2}$, and $P$ is a rational
$d$-polytope. $P$ is said to be $k$-\emph{convex-normal} if the following
equality is satisfied for all $c\in[2,k]_\QQ$:
\begin{align*}
\tag{{\tiny$\CN(d,k)$}}\UU=cP.
\end{align*}
\end{definition}

\medskip Here is a convenient equivalent reformulation. For
$c\in\QQ_{\ge2}$ and $v\in\vertex(P)$ denote by $Q(v)$ the
parallel translate of $(c-1)P$ that moves $(c-1)v$ to $cv$. Put
$$
R(v,c)=\bigcup_{x\in Q(v)\cap\left(cv+\ZZ^d\right)}(x-v+P):
$$
\begin{figure}[htb]
\includegraphics[height=2.4in,width=4.2in]{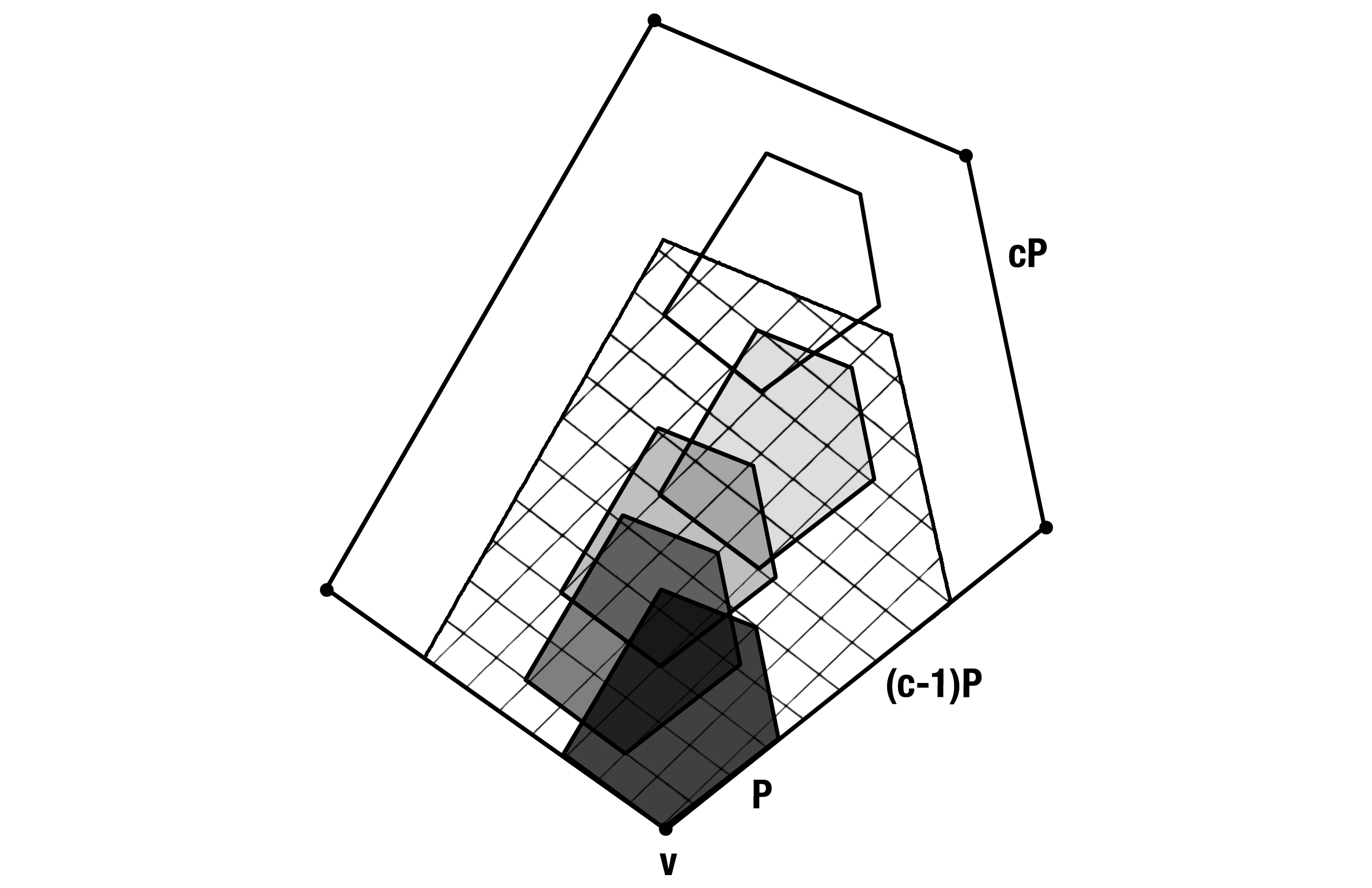}
\end{figure}

\noindent Then $P$ is convex-normal iff for all $c\in[2,k]_\QQ$ we have
$cP=\bigcup_{\vertex(P)} R(v,c)$.

Informally, convex normality is a measure of density of the point configuration $P\cap\ZZ^d$ w.r.t. $P$. For instance, the unimodular simplices of dimension $\ge2$ are \emph{not} convex-normal, but their high multiples are convex-normal. More importantly for our goals, all lattice parallelepipeds are convex-normal; see Lemma \ref{cnormalvsnormal}(a) below.

It is easily observed that a unimodular integral change of coordinates respects the property $\CN(d,k)$, and the same is true for rational parallel translations. Also, one can show (although we do not need it) that $cP=\UU$ for any rational $d$-polytope $P$ and any real number
$c\in\left[1,\frac{d+1}d\right]$.

\begin{lemma}\label{cnormalvsnormal}
(a) Let $\Box$ be a rational parallelepiped. If
$\E(\Box)\ge1$ then $c\Box=\mathbb U_{\vertex}(\Box,c)$ for
every $c\in\QQ_{\ge1}$. If $\E(\Box)<1$ then $\mathbb U_{\vertex}(\Box,c)\not=c\Box$ for all $c\in\mathbb Q_{>2}$, sufficiently close to $2$.

(b)  For every natural number $d$, any $(d-1)$-convex-normal lattice $d$-polytope is integrally
closed.
\end{lemma}

\begin{proof} (a) Assume $\E(\Box)\ge1$. First consider the case $\dim\Box=1$. We can assume $\Box=[0,l]$ for some $l\in\QQ_{\ge1}$. If $c<2$ then
$$
[0,cl]=[0,l]\cup[(c-1)l,cl]\subset\Bbb U_{\vertex}([0,l],c).
$$
If $c\ge2$ then $[0,(c-1)l]\cup[l,cl]=[0,cl]$ and, simultaneously, the inequality $l\ge1$ implies the mutually symmetric inclusions
\begin{align*}
&[0,(c-1)l]\subset\bigcup_{x\in[0,(c-1)l]\cap\ZZ^d}x+[0,l],\\
&[l,cl]\subset\bigcup_{x\in[0,(c-1)l]\cap\big((c-1)l+\ZZ^d\big)}x+[0,l].
\end{align*}

Consider the case $\dim\Box=d>1$. We can assume $\Box\subset\RR^d$. Without loss of generality we can further assume that $\Box=\prod_{i=1}^d[0,l_l]$ for some $l_i\in\QQ_{>1}$. In fact, one first applies a parallel
translation that moves a vertex of $\Box$ to $0$, then applies an
appropriate rational change of coordinates that transforms the primitive lattice edge vectors of $\Box$, emerging from $0$, into the standard basic vectors of $\RR^d$, and, finally, changes the lattice of reference to the integer lattice w.r.t. to the new coordinates. The new lattice is a parallel translate of a subgroup of the old copy of $\ZZ^d$. In particular, $\Bbb U_{\vertex}(\Box,c)$, constructed w.r.t. the `new $\ZZ^d$' is a subset of the one constructed w.r.t. to the `old $\ZZ^d$'. Also, the condition $\E(\Box)\ge1$ remains valid w.r.t to the new lattice of reference.

For $\delta\in\{0,1\}^d$ denote by $v^\delta(\Box)$ the vertex of $\Box$ whose $i$th coordinate is $0$ iff the $i$th component of $\delta$ is $0$. Pick $z=(z_1,\ldots,z_d)\in c\Box$. By the one-dimensional case, for every component $z_i$ we can fix $\delta_i\in\{0,1\}$ so that
$$
z_i\in\bigcup_{\xi\in\left[0,(c-1)l_i\right]\cap\big(\delta_i(c-1)l_i+\ZZ^d\big)}\xi+\left[0,l_i\right].
$$
Then
$$
z\in\bigcup_{x\in(c-1)\Box\cap\big(v^\delta\big((c-1)\Box\big)+\ZZ^d\big)}x+\Box,\qquad \delta=(\delta_1,\ldots,\delta_d).
$$

\medskip Now assume $\E(\Box)<1$. Without loss of generality we can assume $\dim\Box=1$ and, moreover, $\Box=[0,l]$. Pick an arbitrary $\varepsilon\in\QQ_{>0}$ with $\varepsilon<l^{-1}-1$ and let $c=2+\varepsilon$. Then $[0,(c-1)l]\cap\ZZ=\{0\}$ and $[0,(c-1)l]\cap\big((c-1)l+\ZZ\big)=\{(c-1)l\}$, and, consequently,
$$
\frac{cl}2\in[0,cl]\setminus\Bbb U_{\vertex}([0,l],c).
$$

\medskip (b) Notice that lattice segments ($d=1$) and lattice polygons ($d=2$) are vacuously $(d-1)$-convex-normal. So the statement includes
the known fact that all lattice segments and lattice polygons are integrally closed; see \cite[Corollary 2.54]{kripo}.

\medskip Let $P$ be a lattice $d$-polytope. Then
\begin{equation}\label{integralvertex}
v+\ZZ^d=cv+\ZZ^d=\ZZ^d\quad\text{for all}\ v\in\vertex(P)\ \text{and}\ c\in\NN.
\end{equation}

Assume $P$ is a lattice $d$-polytope, satisfying $\CN(d,d-1)$, and
let $c\in[2,d-1]_\NN$. Then, in view of (\ref{integralvertex}), for
every $z\in cP\cap\ZZ^d$ there exist $x\in(c-1)P\cap\ZZ^d$ and
$x_c\in x+P$ with $z=x+x_c$. Then, necessarily, $x_c\in P\cap\ZZ^d$,
and the descending induction from $c$ to $1$ implies
$z=x_1+\cdots+x_c$ with $x_1,\ldots,x_c\in P\cap\ZZ^d$.

Now assume $c\in\NN_{\ge d}$ and $z\in cP\cap\ZZ^d$. Then, by \cite[Theorem 2.52]{kripo} (an essentially equivalent result, but stated for the \emph{normalization of the polytopal monoid} of $P$ instead of the integral closure in $\ZZ^d$, is \cite[Corrolary 1.3.4]{brgutr}),
there exist a natural number $1\le c_0\le d-1$, a lattice point
$x_0\in c_0P\cap\ZZ^d$, and a family of lattice points $x_i\in P\cap\ZZ^d$, $i=1,\ldots,c-c_0$, such that
$z=x_0+x_1+\cdots+x_{c-c_0}$.
So the general case reduces to the case $c\le d-1$.
\end{proof}

\section{$\CN$ in dimension\ $d-1$ $\Longrightarrow$ boundary $\CN$ in dimension $d$}
For a polytope $P$ and a vertex $v\in\vertex(P)$ we let $\FF(P)^v$ denote the facets
$F$ of $P$ that are visible from $v$, i.~e., $v\notin F$.

For two polytopes $Q,P\subset\RR^d$ with $\dim Q=d-1$ and $\dim P=d$ the Euclidean width of $P$ w.r.t. $\aff(Q)$ will be denoted by $\width_Q(P)$.

For a $d$-polytope $P\subset\RR^d$, a facet $F\subset P$, and a real number $\varepsilon>0$ we define the \emph{$\varepsilon$-layer along $F$ inside $P$} to be the
polytope
\begin{align*}
F_P(\varepsilon)=\left\{x\in P\ :\
\|x,H_F\|\le\varepsilon\right\}.
\end{align*}
If $\varepsilon<\width_F(P)$ then $F_P(\varepsilon)$ has a
facet, different from $F$ and parallel to $F$. It will be denoted
by $F_P(\varepsilon)^+$.

\begin{definition}\label{BCN}
Assume $k\in\QQ_{\ge2}$ and $P\subset\RR^d$ is a rational
$d$-polytope. $P$ is said to be $k$-\emph{boundary-convex-normal}
if the following condition is satisfied for every $c\in[2,k]_\QQ$ and every $F\in\FF(P)$:

\begin{equation}
\tag{{\tiny{$\BCN(d,k)$}}}
\begin{aligned}
\big((cF)_{cP}\big)(\varepsilon_F)\subset\UU,\quad \varepsilon_F=\frac{\width_F(P)}{d+1}.
\end{aligned}
\end{equation}
\end{definition}

\begin{lemma}\label{pyramid}
Let  $d\in\NN_{\ge2}$, $k\in\QQ_{\ge2}$, and $\lambda\in\QQ_{>0}$.
Assume every rational $(d-1)$-polytope $Q$ with $\E(Q)\ge\frac
d{d+1}\lambda$ satisfies $\CN\left(d-1,k+\frac{k-1}d\right)$. Let
$P$ be a rational $d$-polytope with $\E(P)\ge\lambda$,
$w\in\vertex(P)$, and $F\in\FF(P)^w$. Then for the rational
$d$-pyramid $\Delta=\conv(w,F)$ and every $c\in[2,k]_\QQ$ we have
\begin{align*}
\big((cF)_{c\Delta}\big)(\varepsilon)\subset\UU,\qquad\varepsilon=\frac{\|w,H_F\|}{d+1}.
\end{align*}
\end{lemma}

\begin{proof}
We can assume $P\subset\RR^d$. Denote:
$$\Pi=\bigcup_{ {\tiny\begin{matrix}
&v\in\vertex(F)\\
&x\in(c-1)F\bigcap\left((c-1)v+\ZZ^d\right)
\end{matrix}
}}x+F_\Delta(\varepsilon).
$$

Since $\Pi\subset\UU$, it is enough to show
\begin{equation}\label{inPi}
\big((cF)_{c\Delta}\big)(\varepsilon)\subset\Pi.
\end{equation}

Let $G=F_\Delta(\varepsilon)^+\in\FF(F_\Delta(\varepsilon))$.
Then $G$ is a homothetic image of $F$ with factor $d/(d+1)$. In
particular, $G$ is a rational $(d-1)$-polytope whose every edge
has lattice length $\ge\frac d{d+1}\lambda$. By the assumption, $G$ satisfies $\CN\left(d-1,k+\frac{k-1}d\right)$.

The rational polytope $K=\big((cF)_{c\Delta}\big)(\varepsilon)^+$ is a
homothetic image of $F$ with factor $\frac{cd+c-1}{d+1}$. So $K$
is a homothetic image of $G$ with factor
$$
c_1=\frac{cd+c-1}{d+1}\cdot\frac{d+1}d=c+\frac{c-1}d\in\left[2+\frac1d,\
k+\frac{k-1}d\right]_\QQ:
$$
\begin{figure}[htb]
\includegraphics[height=3in,width=4.3in,]{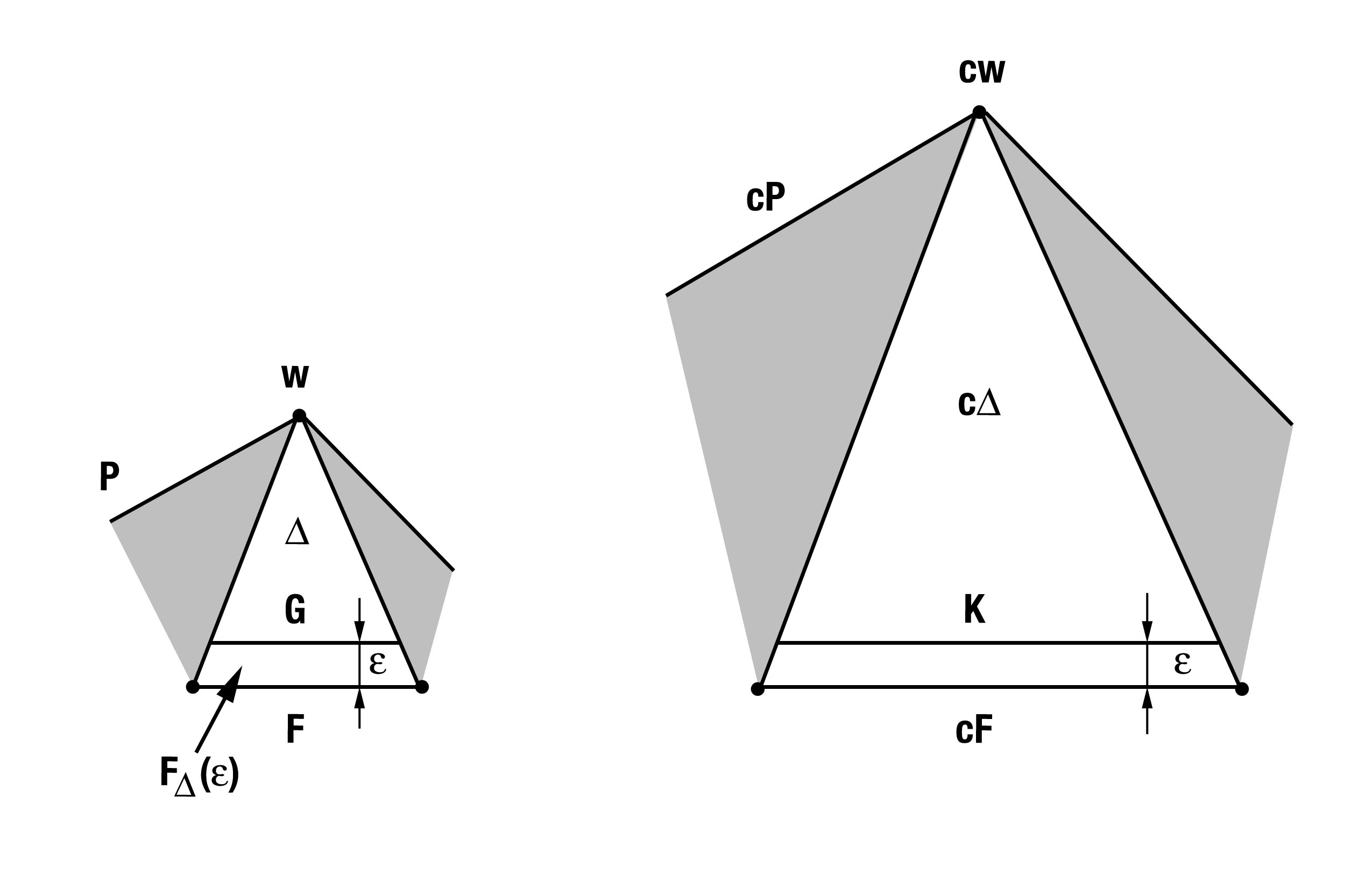}
\end{figure}

\noindent The polytope $(c-1)F$ is a rational homothetic image of $G$ with
factor $\frac{(d+1)(c-1)}d$. In particular, $(c_1-1)G=(c-1)F$ and, by the inductive assumption on rational $(d-1)$ polytopes with lattice edge lengths $\ge\frac d{d+1}\lambda$, we have
\begin{align*}
&\bigcup_{ {\tiny\begin{matrix}
&v\in\vertex(F)\\
&x\in(c-1)F\bigcap\left((c-1)v+\ZZ^d\right)
\end{matrix}
}}x+G=K,
\end{align*}
or, equivalently, $K\subset\Pi$. To put in other words, the lid of
the truncated pyramid $\big((cF)_{c\Delta}\big)(\varepsilon)$ is covered by
the relevant parallel translates of the lid of the smaller
truncated pyramid $F_\Delta(\varepsilon)$.

Pick a point $z\in\big((cF)_{c\Delta}\big)(\varepsilon)$. The ray
$cw+\RR_+(z-cw)$ intersects $K$ at some point $z_K$. Let $z_K\in
x+G$ for some $x$ as in the index set in the definition of $\Pi$. Then
$$
\big(z_K+\RR_+(-w+\Delta)\big)\cap
\big((cF)_{c\Delta}\big)(\varepsilon)=z_K+\frac1{d+1}(-w+\Delta)
$$
and
$$
F_\Delta(\varepsilon)=G+\frac1{d+1}(-w+\Delta).
$$
Therefore,
$$
z\in z_K+\frac1{d+1}(-w+\Delta)\subset
x+G+\frac1{d+1}(-w+\Delta)=x+F_\Delta(\varepsilon)\subset\Pi:
$$
\begin{figure}[htb]
\includegraphics[height=2.5in,width=4.5in,]{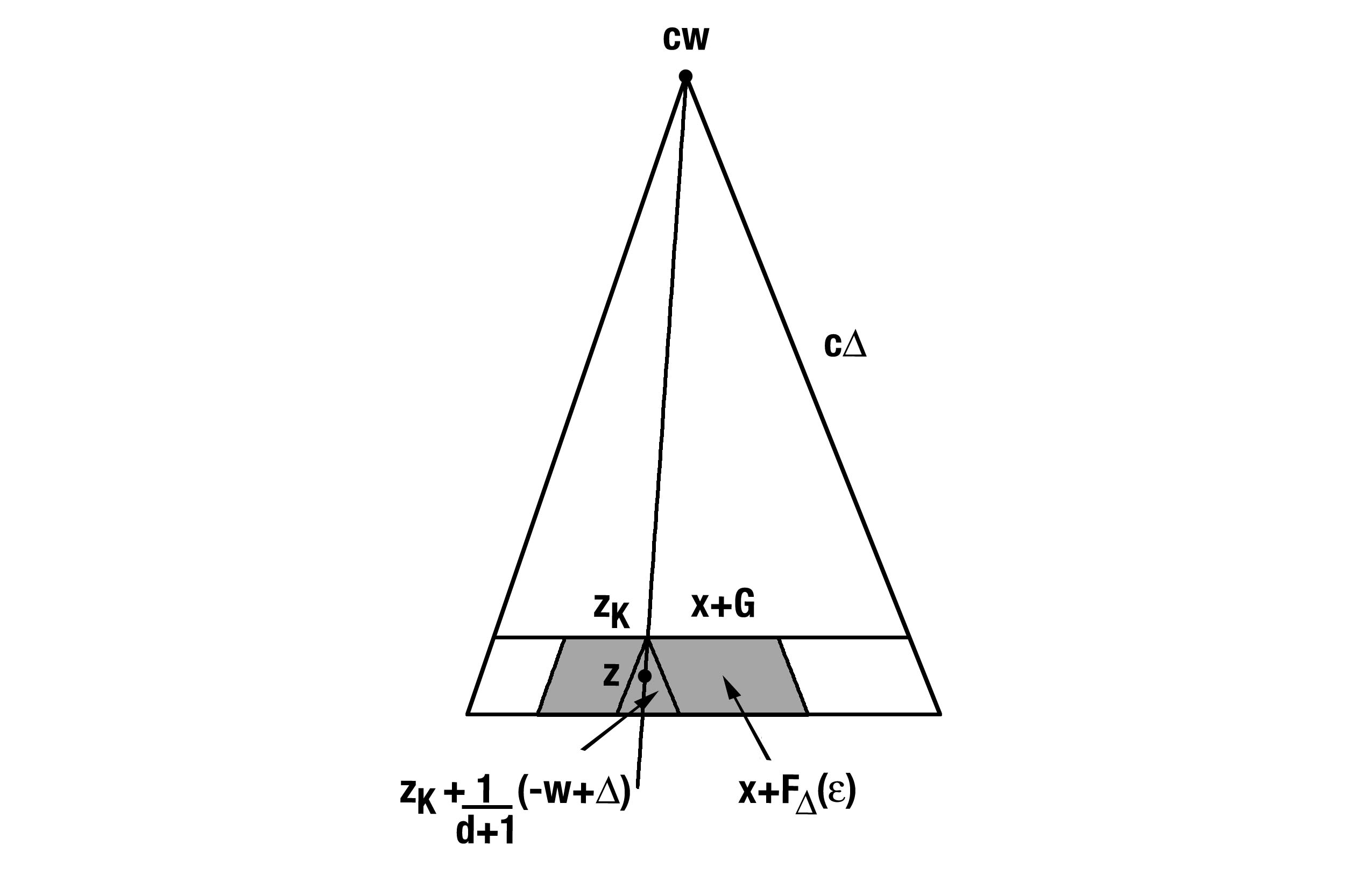}
\end{figure}

\end{proof}

\begin{remark}\label{whyrational}
(a) In the proof of Lemma \ref{pyramid} there are two places that
make it necessary to involve rational polytopes in our induction on dimension: the polytope $G$, to which the assumption on
$(d-1)$-polytopes is applied, is usually not lattice even if $P$
is, and the number $c_1$ is usually not an integer.

(b) If one defined the convex normality by the `dual' equalities:
\begin{align*}
cP=\bigcup_{ {\tiny\begin{matrix}
&v\in\vertex(P)\\
&x\in P\bigcap\left(v+\ZZ^d\right)
\end{matrix}
}} x+(c-1)P\quad\text{for all}\quad&c\in[2,k]_\QQ,
\end{align*}
then the lower bound for the analogue of $c_1$ in the proof of
Lemma \ref{pyramid} would have been $2-\frac1{d+1}$, blocking the
possibility for induction on $d$.
\end{remark}

\begin{lemma}\label{homotheticlayers}
Let  $d\in\NN_{\ge2}$, $k\in\QQ_{\ge2}$, and $\lambda\in\QQ_{>0}$.
If every rational $(d-1)$-polytope $Q$ with $\E(Q)\ge\frac
d{d+1}\lambda$ satisfies $\CN\left(d-1,k+\frac{k-1}d\right)$ then
every rational $d$-polytope $P$ with $\E(P)\ge\lambda$ satisfies
$\BCN(d,k)$.
\end{lemma}

\begin{proof}
Let $P$ be a rational $d$-polytope with edge lengths
$\ge\lambda$, $F\in\FF(P)$, and
$$
\varepsilon_F=\frac{\width_F(P)}{d+1}.
$$
Fix a vertex $w\in\vertex(P)\setminus F$ with
$\|w,H_F\|=\width_F(P)$. Such exists because $\width_F(P)=\max_{\vertex(P)}\big(\|v,H_F\|\big)$.

For every facet $G\in\FF(P)^w$ denote
$$
\Delta(G)=\conv(w,G)\quad\text{and}\quad\varepsilon_{w,G}=\frac{\|w,H_G\|}{d+1}.
$$

By Lemma \ref{pyramid}, for every $c\in[2,k]_\QQ$ we have the inclusion
\begin{align*}
\bigcup_{G\in\FF(P)^w}(cG)_{c\Delta(G)}(\varepsilon_{w,G})\subset\UU.
\end{align*}
But for every $c\in[2,k]_\QQ$ we also have
$$
(cF)_{cP}(\varepsilon_{w,F})=(cF)_{cP}(\varepsilon_F)\subset
cP\setminus\mathsf H(cP))=\bigcup_{G\in\FF(P)^w}(cG)_{c\Delta(G)}(\varepsilon_{w,G}),
$$
where $\mathsf H(cP)$ denotes the homothetic image of $cP$,
centered at $cw$ and with factor $\frac{cd+c-1}{cd+c}$.
The inclusion in the middle, essentially, amounts to the convexity of
$cP$:

\begin{figure}[htb]
\includegraphics[height=3.3in,width=4.3in,]{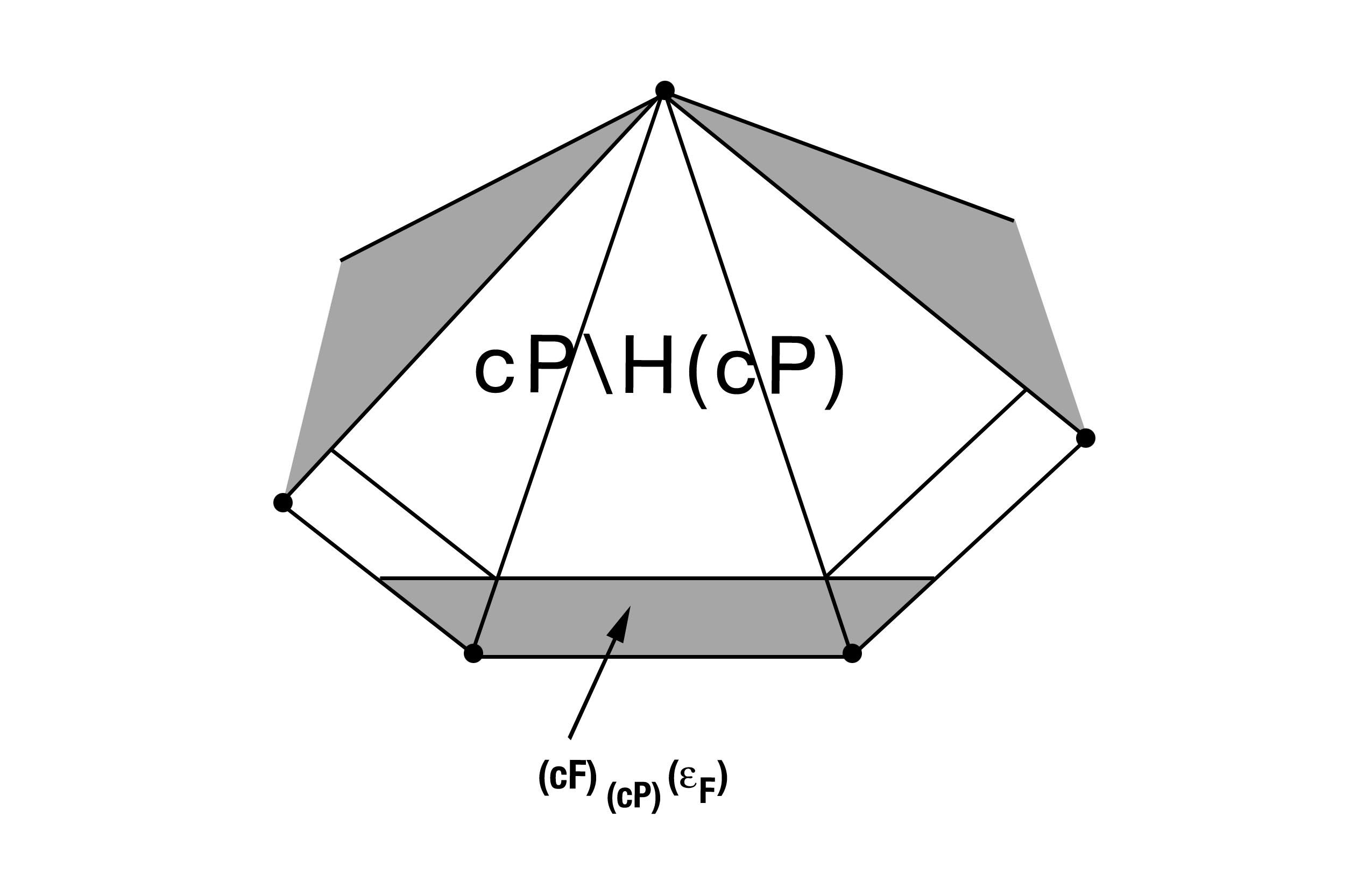}
\end{figure}

\end{proof}

\section{Deep parallelepipedal covers from vertices}\label{deeppcovers}

Fix a rational $d$-polytope $P\subset\RR^d$, a rational number $l\ge1$, and a vertex $v\in P$.

For a system of positive rational numbers
$\bar\varepsilon=\big(\varepsilon_F\big)_{F\in\FF(P)^v}$
we denote
$$
P-\bar\varepsilon\cdot\FF(P)^v=\overline{P\setminus\bigcup_{\FF(P)^v}F_P(\varepsilon_F)},
$$
the `bar' on the right hand side referring to the closure in the Euclidean topology.

\medskip Pick a simplicial $d$-cone of the form
$C=\RR_+(v_1-v)+\cdots+\RR_+(v_d-v)\subset\RR^d$ with $v_1,\ldots,v_d\in\vertex(P)$.

Let $x_i$ be the primitive integer vector in the direction of
$v_i-v$ and $\Box(C)\subset C$ be the
parallelepiped, spanned over $0$ by the $x_i$, $i=1,\ldots,d$. Denote by
$P(\Box(C))$ the union of the integral parallel translates of $\Box(C)$
of type $v+\sum_{i=1}^da_ix_i+\Box(C)$, $a_1,\ldots,a_d\in\ZZ_+$, which fall inside $P$.

\begin{lemma}\label{pcover}
If $\E(P)\ge ld(d+1)$ and $\varepsilon_F=\frac{\width_F(P)}{l(d+1)}$ for every $F\in\FF(P)^v$ then
$$
(P-\bar\varepsilon\cdot\FF(P)^v)\cap C\ \subset\ P(\Box(C)).
$$
\end{lemma}

\begin{proof}
By shifting $P$ by $-v$, we can assume $v=0$.

Pick $x\in\big(P-\bar\varepsilon\cdot\FF(P)^0\big)\cap C$. There exist $b_1,\ldots b_d\in\ZZ_+$ with $x\in\sum_{i=1}^db_ix_i+\Box(C)$. We want to show $\sum_{i=1}^db_ix_i+\Box(C)\subset P$. By the choice of $x$, it is enough to show that $\width_F\big(\Box(C)\big)\le\varepsilon_F$ for every $F\in\FF(P)^0$:

\begin{figure}[htb]
\includegraphics[height=2.8in,width=4.1in,]{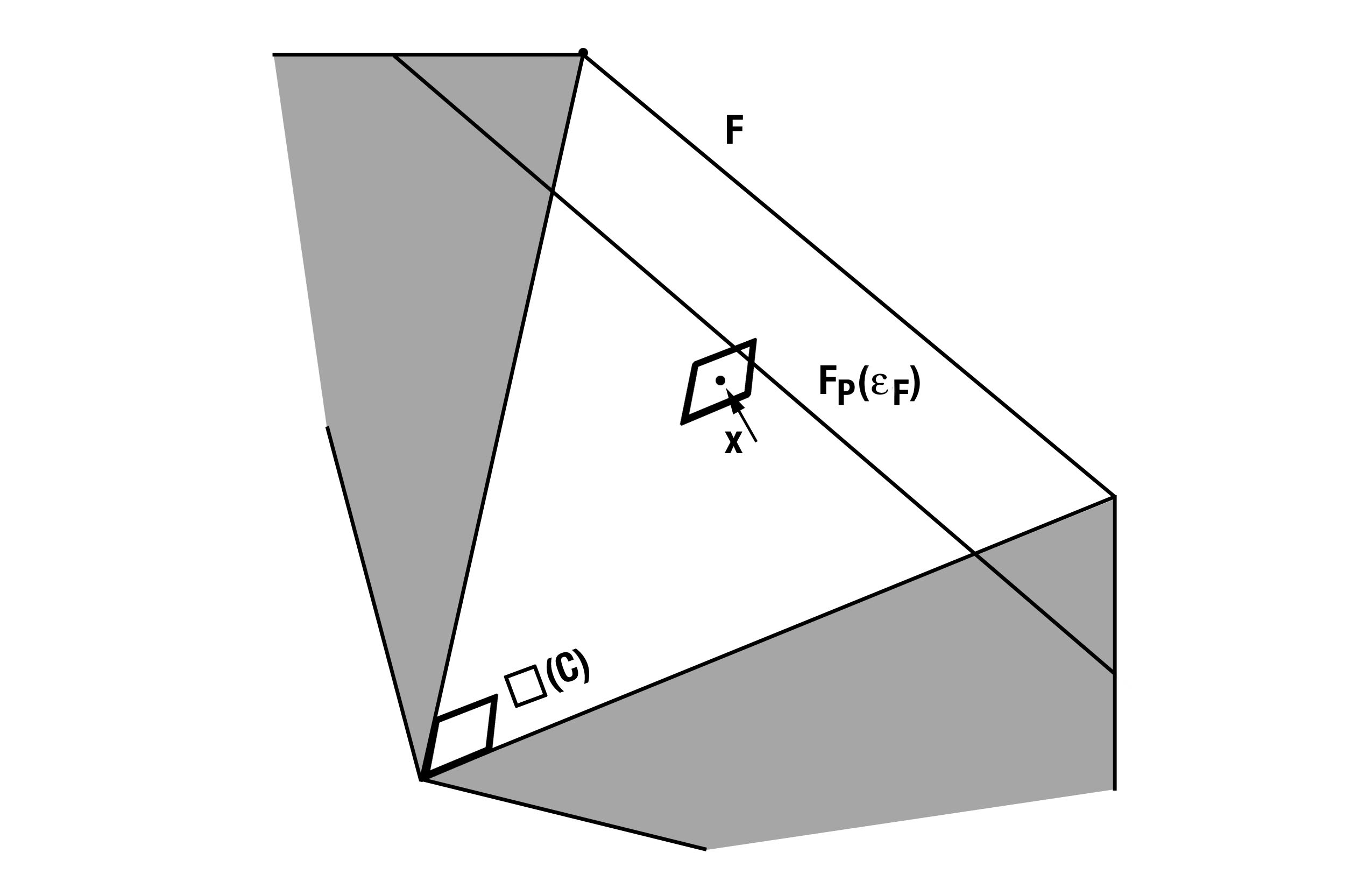}
\end{figure}

Consider the simplices $\Delta_1=\conv(0,x_1,\ldots,x_d)$ and $\Delta_2=\conv(0,v_1,\ldots,v_d)$. Then $\Box(C)\subset d\Delta_1$ and $ld(d+1)\Delta_1\subset\Delta_2$. Therefore, for every $F\in\FF(P)^0$ we have
\begin{align*}
\width_F\big(\Box(C)\big)\le\width_F(d\Delta_1)\le\frac{\width_F(\Delta_2)}{l(d+1)}\le\frac{\width_F(P)}{l(d+1)}=\varepsilon_F.
\end{align*}
\end{proof}

\begin{corollary}\label{pcoverweak}
In the situation of Lemma \ref{pcover}, the union of all lattice parallelepipeds inside $P$ contains
$P-\bar\varepsilon\cdot\FF(P)^v$.
\end{corollary}

\begin{proof}
This follows from Lemma \ref{pcover} and the existence of a cover
of the form $\RR_+(P-v)=\bigcup_JC_j$, where the $C_j\subset\RR^d$, $j\in J$, are
\emph{simplicial} $d$-cones, spanned by extremal generators of the cone
$\RR_+(P-v)$ -- \emph{Carath\'eodory Theorem} for cones; see \cite[Thm.
1.55]{kripo}. One can even choose the cover to be a triangulation
of $\RR_+P$; see \cite[Thm. 1.54]{kripo}, \cite[Prop.
1.15(i)]{ziegler}.
\end{proof}

\section{Recursion rules for $\CN$} Let $d\in\NN$,
$k\in\QQ_{\ge2}$, and $P$ denote a general rational $d$-polytope.
Define:
\begin{align*}
&\cn(d,k)=\inf\big(l\in\QQ\ |\ \E(P)\ge l\ \Longrightarrow\ P\
\text{satisfies}\ \CN(d,k)\big),\\
&\bcn(d,k)=\inf\big(l\in\QQ\ |\ \E(P)\ge l\ \Longrightarrow\ P\
\text{satisfies}\ \BCN(d,k)\big).
\end{align*}

It is not a priori clear that these are finite numbers. What makes
them finite and, in fact, the whole strategy work is the following
recursion rules:

\begin{lemma}\label{inequalities} For $d\in\NN_{\ge2}$ and
$k\in\QQ_{\ge2}$ we have:

\begin{equation}
\tag{a} \cn(1,k)\le1.
\end{equation}
\begin{equation}
\tag{b} \bcn(d,k)\le\frac{d+1}d\cn\left(d-1,k+\frac{k-1}d\right),
\end{equation}
\begin{equation}
\cn(d,k)\le\tag{c} \max\big(kd(d+1),\bcn(d,k)\big),
\end{equation}
\end{lemma}

\begin{proof}[Proof of Lemma \ref{inequalities}] (a) This follows from the first half of Lemma \ref{cnormalvsnormal}(a).

One can say more: $\cn(1,2)=0$ and, by the second half of Lemma \ref{cnormalvsnormal}(a), $\cn(1,k)=1$ for $k>2$.

\medskip(b) This follows from Lemma \ref{homotheticlayers}.

\medskip(c) We will use the following Minkowski sum formula for two homothetic parallelepipeds $\Box_1,\Box_2\subset\RR^d$, with $\Box_1$ at most as large as $\Box_2$:
\begin{equation}\label{paralminksum}
\Box_1+\Box_2=\bigcup_{v\in\vertex(\Box_1)}v+\Box_2
\end{equation}

Let $P\subset\RR^d$ be a rational $d$-polytope with
$\E(P)>\max\big(kd(d+1),\bcn(d,k)\big)$. We want to show that $P$
satisfies $\CN(d,k)$.

Pick $v\in\vertex(P)$. Applying the parallel translation by $-v$,
there is no loss of generality in assuming $v=0$.

Fix a cover of the form $\RR_+P=\bigcup_J C_j$, where the $C_j$,
$j\in J$, are simplicial $d$-cones, spanned by extremal rays of
$\RR_+P$; see the proof of Corollary \ref{pcoverweak}.

Assume $c\in[2,k]_\QQ$. Because $c-1\ge1$ we have
$\E\big((c-1)P\big)\ge\E(P)>kd(d+1)$ and by (twofold application
of) Lemma \ref{pcover}, for every $j\in J$ we have the inclusions:
\begin{equation}\label{why2}
\begin{aligned}
&\left(P-\bar\varepsilon\cdot\FF(P)^0\right)\cap C_j\ \subset\ P(\Box(C_j)),\\
&\left((c-1)P-\bar\varepsilon\cdot\FF\big((c-1)P\big)^0\right)\cap
C_j\ \subset\ \big((c-1)P\big)(\Box(C_j)),
\end{aligned}
\end{equation}
notation as in Lemma \ref{pcover} with $\bar\varepsilon=\left(\varepsilon_F\right)_{\FF(P)^0}$, $\varepsilon_F=\frac{\width_F(P)}{k(d+1)}$.

For $t\in\QQ_{>0}$ denote $t\bar\varepsilon=(t\varepsilon_F)_{F\in\FF(P)^0}$. Because $c-1\ge1$, we have
$$
(c-1)P-(c-1)\bar\varepsilon\cdot\FF\big((c-1)P\big)^0\ \subset\ (c-1)P-\bar\varepsilon\cdot\FF\big((c-1)P\big)^0,
$$
which, together with the second inclusion in (\ref{why2}), gives
\begin{equation}\label{c-1}
\big((c-1)P-(c-1)\bar\varepsilon\cdot\FF\big((c-1)P\big)^0\big)\cap
C_j\ \subset\ \big((c-1)P\big)(\Box(C_j)).
\end{equation}

Pick $j\in J$. Denote by $A$, resp. by $B$, the set of parallelepipeds of type
\begin{align*}
\sum_{i=1}^da_ix_{ji}+\Box(C_j),\qquad &a_1,\ldots,a_d\in\ZZ_+,\\
&x_{j1},\ldots,x_{jd}\quad \text{-- the extremal
generators of}\quad C_j,
\end{align*}
which fall inside $(c-1)P$, resp. inside $P$. Then we have

\begin{align*}
&\big((c-1)P\big)(\Box(C_j))+P(\Box(C_j))=\\
\\
&\qquad\bigcup_{\left(\Box_1,\Box_2\right)\in{A\times
B}}\Box_1+\Box_2=\bigcup_{\tiny{\begin{matrix} \Box_1\in A\\
x\in\vertex(\Box_1)
\end{matrix}}}\bigcup_{\Box_2\in B}x+\Box_2=\\
\\
&\quad\quad\quad\quad\quad\quad\bigcup_{x\in((c-1)P)(\Box(C_j))\cap\ZZ^d}\ \bigcup_{\Box\in B}x+\Box=\bigcup_{x\in((c-1)P)(\Box(C_j))\cap\ZZ^d}x+P(\Box(C_j))\subset\\
\\
&\quad\quad\qquad\qquad\qquad
\bigcup_{x\in((c-1)P)(\Box(C_j))\cap\ZZ^d}x+\big(P\cap C_j\big)\subset\\
&\quad\quad\qquad\qquad\qquad\qquad\qquad\qquad
\bigcup_{x\in((c-1)P)\cap\ZZ^d}x+\big(P\cap C_j\big),
\end{align*}
where the second and third equalities follow from (\ref{paralminksum}). We record:
\begin{equation}\label{crucialsequence}
\big((c-1)P\big)(\Box(C_j))+P(\Box(C_j))\subset\bigcup_{x\in((c-1)P)\cap\ZZ^d}x+\big(P\cap C_j\big).
\end{equation}

On the other hand, for every $j\in J$, the following equality holds true for reasons of homothety
(w.r.t. to the origin):
\begin{equation}\label{homothetically}
\begin{aligned}
\left((c-1)P-(c-1)\bar\varepsilon\cdot\FF\big((c-1)P\big)^0\right)\cap
C_j\ +\ &\left(P-\bar\varepsilon\cdot\FF(P)^0\right)\cap
C_j=\\&\left(cP-c\bar\varepsilon\cdot\FF(cP)^0\right)\cap C_j.
\end{aligned}
\end{equation}
Then, integrating over $j\in J$, the first
inclusion in (\ref{why2}), (\ref{c-1}), (\ref{crucialsequence}), and
(\ref{homothetically}) imply
\begin{equation}\label{deep}
cP-c\bar\varepsilon\cdot\FF(cP)^0\ \subset\ \bigcup_{x\in(c-1)P\cap\ZZ^d}x+P.
\end{equation}

For every $F\in\FF(P)^0$ we have
$c\varepsilon_F\le\frac{\width_F(P)}{d+1}$. Therefore,
\begin{equation}\label{14}
cP-\bar\sigma\cdot\FF(cP)^0\ \subset\ cP-c\bar\varepsilon\cdot\FF(cP)^0,
\end{equation}
where
$$
\bar\sigma=(\sigma_F)_{F\in\FF(P)^0},\qquad\sigma_F=\frac{\width_F(P)}{d+1}.
$$
Because $\E(P)>\bcn(d,k)$, (\ref{deep}) and (\ref{14}) together
imply $\CN(d,k)$ for $P$.
\end{proof}

\begin{corollary}\label{maininequality}
(a) For all  $d\in\NN_{\ge2}$ and $k\in\QQ_{\ge2}$ we have
$$
\cn(d,k)\le\max\left(d(d+1)k,\
\frac{d+1}d\cn\left(d-1,k+\frac{k-1}d\right)\right).
$$

(b) For all $d\in\NN$ and $k\in\QQ_{\ge2}$ we have
$\cn(d,k)<\infty$.
\end{corollary}

The part (a) follows from Lemma \ref{inequalities}(b,c), and the part (b)
follows from the part (a) and Lemma \ref{inequalities}(a).

\begin{remark}\label{finalremark}
(a) In the proof above we used twice that $c-1\ge1$. This explains
why in Definition \ref{definition} we choose $k\ge2$ and
$c\in[2,k]_\QQ$, and not $k\ge1$ and $c\in[1,k]_\QQ$.

(b) We have \emph{not} shown that $\lim_{k\to\infty}\cn(d,k)<\infty$.
\end{remark}

\section{Proof of the main result}
\subsection{Proof of Theorem \ref{conjecture(a)}}
The limit case will be taken care of by

\begin{lemma}\label{contraction}
Let $d\in\NN$ and $k\in\QQ_{\ge2}$. If $P$ is a rational
$d$-polytope with $\E(P)=\cn(k,d)$ then $P$ satisfies
$\CN(d,k)$.
\end{lemma}

\begin{proof}
We can assume $P\subset\RR^d$. On the one
hand, for  all $c\in[2,k]_\QQ$ and all sufficiently small
$\varepsilon\in\QQ_{>0}$, depending on $k$ and $P$ (but not on $c$), the following holds true for any vertex $v\in P$: the set
$$
(c-1)(1+\varepsilon)P\ \cap\
\left((c-1)(1+\varepsilon)v+\ZZ^d\right)\subset\RR^d
$$
is the parallel translate by $\varepsilon(c-1)v$
of the set
$$
(c-1)P\ \cap\ \left((c-1)v+\ZZ^d\right)\subset\RR^d.
$$
On the other hand, the polytope $(1+\varepsilon)P$ is a homothetic image of $P$, approximating $P$ as $\varepsilon\to 0$.
Consequently, since the unions of only finitely many polytopes are involved, for every number $c\in[2,k]_\QQ$, the complement $cP\setminus\UU$ is a closed measurable set
in $\RR^d$ that can be approximated measure-wise with arbitrary precision by sets of the form
$c(1+\varepsilon)P\setminus\Bbb U_{\vertex}\big((1+\varepsilon)P,c\big)$, $\varepsilon\in\QQ_{>0}$. But the latter are all empty sets.
\end{proof}

Now we turn to Theorem \ref{conjecture(a)} proper. By Corollary
\ref{maininequality}(b), the function
$\cn(d,k):\NN\times\QQ_{\ge2}\to\RR_+$ is well defined. For any
fixed $d\in\NN$ the function $\cn(d,k):\QQ_{\ge2}\to\RR_+$ is
non-decreasing. So, by Corollary \ref{maininequality}(a), for all $d\in\NN_{\ge2}$ and $k\in\QQ_{\ge2}$ we have
the (simpler) inequalities:
$$
\cn(d,k)\le\max\left(d(d+1)k,\ \
\frac{d+1}d\cn\left(d-1,\frac{d+1}d\cdot k\right)\right).
$$
By induction on $i$, based on iterative use of this inequality, we derive
\begin{align*}
\cn(d&,k)\le\\
&\max_{i=1,
\ldots,d-1}\left(\left\{\frac{d+1-j}{d+2-j}\cdot(d+1)^2k\right\}_{j=1}^i,\
\ \frac{d+1}{d+1-i}\cn\left(d-i,\ \frac{d+1}{d+1-i}\cdot k\right)\right).
\end{align*}
Therefore,
\begin{align*}
&\cn(d,k)\le\\
&\qquad\max\left(\left\{\frac{d+1-j}{d+2-j}\cdot(d+1)^2k\right\}_{j=1}^{d-1},\
\ \frac{d+1}2\cn\left(1,\ \frac{d+1}2\cdot k\right)\right)\le\\
&\qquad\qquad\qquad\max\left(d(d+1)k,\
\frac{d+1}2\right)=d(d+1)k.
\end{align*}

This already proves the version of Theorem \ref{conjecture(a)}
with the strict inequality $\E(P)>d(d+1)k$, and the
non-strict inequality is covered by Lemma \ref{contraction}.\qed

\subsection{Proof of Theorem \ref{latticepolytopes}(a)}\label{subsection}
All we need is

\begin{lemma}
Every lattice $d$-polytope $P$ with $\E(P)\ge\cn(d,4)$ is
integrally closed.
\end{lemma}

\begin{proof}
Let $P\subset\RR^d$ be as in the lemma. We show the equality in
Definition \ref{definition}(a) by induction on the factors $c\in\NN$. Assume
it has been shown for all factors $<c$.

For every $n\in\NN$ denote
\begin{align*}
I_n=\left[2^n,2^{n+1}\right]_\NN,\qquad P_n=2^{n-1}P,\qquad
L_n=2^{n-1}\ZZ^d\subset\ZZ^d.
\end{align*}
Then $P_n$ is a rational polytope with $\E_n(P)\ge\cn(d,4)$, where the subindex in $\E_n$ indicates that the lattice lengths are measured w.r.t. $L_n$.

Let $c\in I_n$ for some $n\in\NN$, and pick $z\in cP\cap\ZZ^d$. We
have
$$
cP=
\begin{cases}
c'P_n\ \text{with}\ c'=c2^{-n+1}\in[2,4]_\QQ\ \text{if}\ n>1,\\
c\in[2,4]_\QQ\ \text{if}\ n=1.
\end{cases}
$$

If $n>1$ then $P_n$ satisfies $\CN(d,4)$ w.r.t. the lattice $L_n$; one invokes Lemma
\ref{contraction} in the limit case $\E(P)=\cn(d,4)$. So $z=x+y$ for some
$x\in(c'-1)P_n\cap\left((c'-1)v+L_n\right)$, $v\in\vertex P_n$,
and $y\in P_n$. Then, necessarily, $y\in P_n\cap\ZZ^d$. In
particular,
$\left((c-2^{n-1})P\right)\cap\ZZ^d\ +\ \left(2^{n-1}P\right)\cap\ZZ^d=(cP)\cap\ZZ^d$.

If $n=1$ then we have $z\in\big((c-1)P\big)\cap\ZZ^d\ +\ P$; again, Lemma
\ref{contraction} is invoked in the limit case $\E(P)=\cn(d,4)$. Therefore,
$((c-1)P)\cap\ZZ^d\ +\ P\cap\ZZ^d=(cP)\cap\ZZ^d$.

In both cases the induction assumption applies.
\end{proof}

\subsection{Proof of Theorem \ref{latticepolytopes}(b)} Lattice parallelepipeds are integrally closed -- a consequence of Lemma
\ref{cnormalvsnormal}(a). Therefore, we only need to show that a rational simplex $P$ with $\E(P)\ge d(d+1)$
is covered by lattice parallelepipeds. In view of Corollary
\ref{pcoverweak}, it is enough to show that we have the cover
$$
\bigcup_{v\in\vertex(P)}P-\bar\varepsilon(v)\cdot\FF(P)^v=P,
$$
where $\bar\varepsilon(v)=\left(\varepsilon_F\right)_{F\in\FF(P)^v}$
for every $v\in\vertex(P)$ and
$\bar\varepsilon_F=\frac{\width_F(P)}{d+1}$
for every $F\in\FF(P)$. (Notation as in that corollary.)

Since $P$ is a simplex, for every
vertex $v\in\vertex(P)$ the polytope $P-\bar\varepsilon\cdot\FF(P)^v$ is the
homothetic image of $P$ with factor $\frac d{d+1}$ and centered at
$v$. Therefore, the desired covering follows from the fact that at least
one of the barycentric coordinates of each point $x\in P$ w.r.t.
the vertices of $P$ is $\ge\frac 1{d+1}$.\qed

\begin{remark}\label{whysimplices}
(a) The equality $\bigcup_{v\in\vertex(P)}P-\bar\varepsilon(v)\cdot\FF(P)^v=P$
does not hold true for general polytopes, not even in dimension 2. This explains the need of $\BCN(d,k)$ in the proof of
Theorem \ref{conjecture(a)}.

(b) We have the following minor improvement of Theorem \ref{conjecture(a)} in dimensions $d=3,4$: every lattice $d$-polytope $P$ with $\E(P)\ge d(d^2-1)$ is integrally closed. In fact,Theorem \ref{conjecture(a)} and Lemma \ref{cnormalvsnormal}(b) imply the version of
Theorem \ref{latticepolytopes}(a) with the inequality
$\E(P)\ge d(d^2-1)$, which is a better estimate than $\E(P)\ge 4d(d+1)$ for $d=3,4$.
\end{remark}

\noindent\emph{Notice.} The results in this paper extend to all
polytopes whose edges are parallel to rational directions and all
real factors $\ge2$. For the approach developed above, the most general setting possible is when one fixes an arbitrary
finitely generated additive subgroup $\Lambda\subset\RR^d$ (no longer a discrete subset of $\RR^d$ if $\rank\Lambda>d$) and
studies polytopes $P\subset\RR^d$ whose edge directions
are parallel to elements of $\Lambda$.

A notable exception from the arguments above that go through when $\rank\Lambda>d$ is the proof of Lemma \ref{contraction}.

\bibliography{bibliography}
\bibliographystyle{plain}

\end{document}